\documentclass{amsart}
\usepackage{amssymb}
\usepackage{amsmath}
\usepackage{thmtools, thm-restate}
\usepackage{hyperref}
\hypersetup{
     colorlinks=true,
     linkcolor=blue,
     citecolor = blue,      
     }
\usepackage{cleveref}
\crefname{lem}{le lemme}{les lemmes}
\newtheorem{thm}{Theorem}[section]
\newtheorem{prop}[thm]{Proposition}
\newtheorem{lem}[thm]{Lemma}
\newtheorem{remark}[thm]{Remark}

\theoremstyle{definition}
\newtheorem{definition}[thm]{Definition}

\numberwithin{equation}{section}
\newcommand{\m}{{\rm{m}}}
\newcommand{\M}{{\rm{M}}}

\newcommand{\F}{{\rm{F}}}

\begin{document}

\title{Triple lines on a cubic threefold}
\author{Gloire Grâce Bockondas}
\address{Gloire Grâce Bockondas, Département de Mathématiques, Université Marien Ngouabi, Brazzaville, Congo}
\email{gloire.bockondas@umng.cg}
\urladdr{https://sites.google.com/view/gloiregbockondas/} 
\author{Samuel Boissière}
\address{ Samuel Boissière, Laboratoire de Mathématiques et Applications, UMR 7348 du CNRS,
Bâtiment H3, Boulevard Marie et Pierre Curie, Site du Futuroscope, TSA 61125, 86073 Poitiers Cedex 9, France}
\email{samuel.boissiere@math.univ-poitiers.fr}
\urladdr{http://www-math.sp2mi.univ-poitiers.fr/~sboissie/} 

\begin{abstract}
The present paper deals with lines contained in a smooth complex cubic threefold. It is well-known that the set of lines of the second type on a cubic threefold is a curve on its Fano surface. Here we give a description of the singularities of this curve.
\end{abstract}

\keywords{cubic threefold, Fano surface, triple line.}
\subjclass{14J10; 14J29; 14H20}

\maketitle

\section{Introduction}

Let $X\subset\mathbb{P}^{4}$ be a smooth complex cubic threefold and $\ell\subset X$ a line. The variety that parametrizes the lines on $X$ is a smooth surface of general type called the Fano surface $\F(X)$ of $X$ (see \cite{roulleau2011fano}). Cubic threefolds and their Fano surfaces of lines is an object of interest in algebraic geometry going back to Fano's work in 1904. Later on, their geometry has been further  studied by Clemens and Griffiths who proved the irrationality of the cubic threefold by introducing the intermediate Jacobian as a principal tool \cite{clemens1972intermediate}, Murre studied the geometry of the Fano variety of lines of a smooth cubic threefold \cite{murre1972algebraic}, and Roulleau gave the classification of the configurations of the elliptic curves on the Fano surface of a smooth cubic threefold \cite{roulleau2009elliptic}. We can also cite Altman and Kleiman work \cite{altman1977foundations} and Tjurin papers \cite{tjurin1971geometry, tjurin1972set}. In this work we are interested in lines contained in $X$. They are either of the first type or of the second type depending on the decomposition of the normal bundle $\mathcal{N}_{\ell/X}$. Lines of the first type are generic while the locus of lines of the second type defines a curve $\M(X)$ in the associated Fano surface $\F(X)$ (see \cite{murre1972algebraic}). In the curve $\M(X)$ of lines of the second type there may be particular lines $\ell\subset X$ for which there exists a unique 2-plane $P\supset \ell$ tritangent to $X$ along all of $\ell$ called triple lines. But this only occurs for at most a finite number of lines (see \cite{clemens1972intermediate}). By using local computations Murre proved that the curve $\M(X)$ of lines of the second type is smooth \cite{murre1972algebraic}. However, it is well-known to the experts that this curve can be singular and its singularities correspond to triple lines on $X$ (see \cite{lahoz2021geometry}, \cite{huybrechts2019geometry}, \cite{roulleau2011fano}). Nevertheless, no complete proof does appear in the literature. The purpose of this paper is to fill this gap by giving a complete proof of Theorem \ref{th1} following the techniques of \cite{murre1972algebraic}.

\begin{restatable}[]{thm}{Triplelines}\label{th1}
The triple lines on a cubic threefold are exactly the singular points of the curve of lines of the second type.
\end{restatable}

Consider the Fermat cubic in $\mathbb{P}^{4}$ defined by
$$ F_{4} = \lbrace x_{0}^{3} + x_{1}^{3} + x_{2}^{3} + x_{3}^{3} + x_{4}^{3} = 0 \rbrace.$$ 

\begin{restatable}[]{prop}{Fermat}
The Fermat cubic $F_{4}\subset\mathbb{P}^{4}$ contains exactly 135 triple lines.
\end{restatable}
\textbf{Acknowledgements}.
We would like to thank  Xavier Roulleau, Daniel Huybrechts for helpful discussions, and Mart\'{i} Lahoz, Juan Carlos Naranjo and Andr\'{e}s Rojas for their comments. The first author has been supported by the Program EMS SIMONS for Africa and the ``Laboratoire de Mathématiques et Applications de l'Université de Poitiers UMR CNRS 7348".

\section{Notation and preliminaries}

For $X\subset\mathbb{P}^{4}$ a smooth complex cubic threefold, we denote by $\F(X)$ the Fano surface of lines and $\M(X)$ the curve of lines of the second type. Denote by $(x_{0}: x_{1}: x_{2}: x_{3}: x_{4})$ the homogeneous coordinates on $\mathbb{P}^{4}$ and $p_{i,j}, 0\leq i < j \leq 4$ the Plücker coordinates of the grassmannian of lines $\mathbb{G}(1,4)\subset \mathbb{P}^{9}$. Throughout this paper, we locally study the Fano surface of lines $\F(X)$. For instance when we work in the affine chart $p_{0,1}=1$ of $\mathbb{G}(1,4)$ isomorphic to $\mathbb{C}^{6}$, a point $[\ell]\in\F(X)\subset\mathbb{G}(1,4)$ corresponds to a line on $X$ spanned by two points $v_{0}=(1:0:-p_{1,2}:-p_{1,3}:-p_{1,4})$ and $v_{1}=(0:1:p_{0,2}:p_{0,3}:p_{0,4})$.
We take $(p_{0,2}, p_{0,3}, p_{0,4}, p_{1,2}, p_{1,3}, p_{1,4})$ as the local
coordinates of $\mathbb{G}(1,4)$ on the affine chart $p_{0,1}=1$
(see \cite{boissiere2019cubic}).\newline
Let $\ell$ be a line on $X$. There are two types of lines in $X$: lines with $\mathcal{N}_{\ell/ X}\simeq \mathcal{O}_{\ell}\oplus \mathcal{O}_{\ell}$ called lines of the first type and those with $\mathcal{N}_{\ell/ X}\simeq \mathcal{O}_{\ell}(1)\oplus \mathcal{O}_{\ell}(-1)$ called lines of the second type \cite{clemens1972intermediate}. Another description is given by the following proposition: 
\begin{prop}\cite{clemens1972intermediate}
The line $\ell\subset X$ is of the second type if and only if there exists a unique 2-plane $P\supset \ell$ tangent to $X$ at every point
of $\ell$. If $\ell\subset X$ is a line of the first type, then there is no 2-plane tangent to $X$ in all points of $\ell$.
\end{prop}

\section{Lines on a cubic threefold}

Consider the intersection $P\cap X$ of the cubic threefold $X$ with a plane $P\simeq \mathbb{P}^{2}$ containing $\ell$. We observe that $P\cap X = \ell \cup C$ where $C$ is a conic. It may happen that the conic degenerates, that is $P\cap X = \ell\cup \ell^{'}\cup \ell^{''}$. If $P\cap X = 2\ell \cup \ell^{'}$ then we say the plane $P$ is tangent to $X$ at every point of $\ell$ \cite{tjurin1971geometry}. The line $\ell$ is called a multiple line and $\ell^{'}$ the residual line. We have in particular two cases: if $\ell\neq \ell^{'}$ then $\ell$ is called a double line, and if $\ell=\ell^{'}$, that is $P\cap X=3\ell$, then $\ell$ is called a triple line. The mutiple lines on $X$ are exactly the lines of the second type \cite{murre1972algebraic}.

\subsection{Lines of the second type}

\begin{definition}
Lines $\ell\subset X$ for which there exists a unique 2-plane $P\supset\ell$ such that $P\cap X = 2\ell \cup \ell^{'}$ are called lines of the second type.
\end{definition}

\begin{prop} \cite{murre1972algebraic}
The locus $$\M(X)=\lbrace [\ell]\in \F(X),~\exists~P\simeq\mathbb{P}^{2}~\vert~P\cap X=2\ell \cup \ell^{'}\rbrace$$
of lines of the second type on $X$ is a curve.
\end{prop}

On the affine chart $p_{0,1}=1$, for any point $p\in\ell\subset X$ with coordinates $t_{0}v_{0}+t_{1}v_{1}$ the Fano surface $\F(X)$ is defined by the decomposition
$F(p)=\displaystyle\sum_{i+j=3}t_{0}^{i}t_{1}^{j}\phi^{i,j}(\ell)$ 
where $\phi^{i,j}(\ell)$  are functions of the local Plücker coordinates. Following \cite{boissiere2019cubic}, any plane $P\simeq\mathbb{P}^{2}$ containing $\ell$ meets the plane $P_{0,1}=\lbrace x_{0}=0, x_{1}=0\rbrace$ at a unique point $v_{2}=(0:0:\alpha_{2}:\alpha_{3}:\alpha_{4})$ so that $P={\rm{span}}(\ell,v_{2})$. Then the plane cubic $P\cap X$ has equation $F(t_{0}v_{0}+t_{1}v_{1}+t_{2}v_{2})=0$ where $(v_{0}:v_{1}:v_{2})$ are the projective coordinates of the plane $P$. Expanding in $t_{2}$ we have:
\begin{eqnarray*}
0 &=& F(t_{0}v_{0}+t_{1}v_{1})+t_{2} \sum_{i=2}^{4}\dfrac{\partial F}{\partial x_{i}}(t_{0}v_{0}+t_{1}v_{1})\alpha_{i}\\
&+&\dfrac{1}{2}t_{2}^{2}\sum_{2\leq i,j\leq 4}\dfrac{\partial^{2} F}{\partial x_{j}\partial x_{i} }(t_{0}v_{0}+t_{1}v_{1})\alpha_{i}\alpha_{j}+t_{2}^{3}F(v_{2}).
\end{eqnarray*}
The line $\ell\subset P$ of equation $t_{2}=0$ is a line of the second type on $X$ if and only if $F(t_{0}v_{0}+t_{1}v_{1})=0$ and the plane cubic equation is a multiple of $t_{2}^{2}$. For any point $p\in\ell\subset X$ with coordinates $t_{0}v_{0}+t_{1}v_{1}$ on the affine chart $p_{0,1}=1$, one can write 
\begin{equation}\label{NUM1}
\dfrac{\partial F}{\partial x_{i}}(p)=\sum_{j+k=2}t_{0}^{j}t_{1}^{k}\phi_{i}^{j,k}(\ell)
\end{equation}
where $\phi_{i}^{j,k}(\ell)$ are functions of the local Plücker coordinates. Then the line $\ell\subset X$ of equation $t_{2}=0$ is a second type line if and only if
\begin{equation}\label{det}
\det \begin{pmatrix}
\phi^{2,0}_{2}(\ell)& \phi^{2,0}_{3}(\ell) & \phi^{2,0}_{4}(\ell) \\ 
\phi^{1,1}_{2}(\ell) & \phi^{1,1}_{3}(\ell) & \phi^{1,1}_{4}(\ell) \\ 
\phi^{0,2}_{2}(\ell) & \phi^{0,2}_{3}(\ell) & \phi^{0,2}_{4}(\ell) \\
\end{pmatrix}=\m(\ell)=0.
\end{equation}
The curve of lines of the second type on $X$ is thus locally given by $$\M(X)=\lbrace \phi^{3,0}(\ell)=0, \phi^{2,1}(\ell)=0, \phi^{1,2}(\ell)=0, \phi^{0,3}(\ell)=0, \m(\ell)=0\rbrace.$$

\subsection{Triple lines}

\begin{definition}
Lines $\ell\subset X$ of the second type so that $P\cap X = 3\ell$ are called triple lines.
\end{definition}

To our knowledge, the locus of triple lines on a cubic threefold is shortly mentioned in the literature. We have the following lemma.
\begin{lem}\cite{clemens1972intermediate}
The locus $\lbrace [\ell]\in \F(X),~\exists~P\simeq\mathbb{P}^{2}~\vert~P\cap X=3\ell\rbrace$
of triple lines on $X$ defines a finite set.
\end{lem}

The line of the second type $\ell\subset X$ of equation $t_{2}=0$ is a triple line if and only if the plane cubic equation is a multiple of $t_{2}^{3}$. For all point $p\in\ell\subset X$ with coordinates $t_{0}v_{0}+t_{1}v_{1}$ one can write 
\begin{equation}\label{NUM2}
\dfrac{\partial^{2}F}{\partial x_{i}\partial x_{j}}(p)=\sum_{k+l=1}t_{0}^{k}t_{1}^{l}\phi_{i,j}^{k,l}(\ell)
\end{equation}
where $\phi_{i,j}^{k,l}(\ell)$ are functions of the local Plücker coordinates on the affine chart $p_{0,1}=1$. Then the line of the second type $\ell\subset X$ of equation $t_{2}=0$ is a triple line if and only if
\begin{equation}\label{droitetriple}
\begin{pmatrix}
\dfrac{1}{2}\phi_{2,2}^{1,0}(\ell) & \dfrac{1}{2}\phi_{3,3}^{1,0}(\ell) & \dfrac{1}{2} \phi_{4,4}^{1,0}(\ell)  & \phi^{1,0}_{2,3}(\ell) & \phi^{1,0}_{2,4}(\ell) & \phi^{1,0}_{3,4}(\ell) \\ 
\\
\dfrac{1}{2}\phi_{2,2}^{0,1}(\ell) & \dfrac{1}{2}\phi_{3,3}^{0,1}(\ell) & \dfrac{1}{2} \phi_{4,4}^{0,1}(\ell) & \phi^{0,1}_{2,3}(\ell) & \phi^{0,1}_{2,4}(\ell) & \phi^{0,1}_{3,4}(\ell) 
\end{pmatrix} 
\begin{pmatrix}
\alpha_{2}^{2}\\
\alpha_{3}^{2}\\
\alpha_{4}^{2}\\
\alpha_{2}\alpha_{3}\\
\alpha_{2}\alpha_{4}\\
\alpha_{3}\alpha_{4}\\
\end{pmatrix} 
=0
\end{equation}
holds.

\section{Triple lines on a cubic threefold and the curve of lines of the second type}

In this section we prove the following theorem by using the techniques of \cite{murre1972algebraic}.\Triplelines*

\begin{proof}

Let us start by fixing the notations. We can assume that $[\ell_{0}]\in\M(X)$ is a line of the second type given by $$x_{2}=0, x_{3}=0, x_{4}=0.$$ On the affine chart $p_{0,1}=1$ of $\mathbb{G}(1,4)$ the line $[\ell_{0}]\in\M(X)$ is thus given by 
$$p_{0,2}=0,p_{0,3}=0,p_{0,4}=0,p_{1,2}=0,p_{1,3}=0,p_{1,4}=0.$$
We may assume after a linear change of variables $x_{2},x_{3}$ and $x_{4}$ that the 2-plane $P={\rm{span}}(\ell_{0},v_{2})$ tangent to $X$ at every point of $\ell_{0}$ is such that $\alpha_{2}=0, \alpha_{3}=0$ and $\alpha_{4}=1$. Then from the following equation
\begin{equation*}
\begin{pmatrix}
\phi^{2,0}_{2}(\ell_{0})& \phi^{2,0}_{3}(\ell_{0}) & \phi^{2,0}_{4}(\ell_{0}) \\ 
\phi^{1,1}_{2}(\ell_{0}) & \phi^{1,1}_{3}(\ell_{0}) & \phi^{1,1}_{4}(\ell_{0}) \\ 
\phi^{0,2}_{2}(\ell_{0}) & \phi^{0,2}_{3}(\ell_{0}) & \phi^{0,2}_{4}(\ell_{0}) \\
\end{pmatrix} 
\begin{pmatrix}
\alpha_{2}\\
\alpha_{3}\\
\alpha_{4}
\end{pmatrix} 
=\begin{pmatrix}
0 \\ 
0\\ 
0\\
\end{pmatrix}
\end{equation*}
we get
\begin{equation}\label{eq:phi4}
\phi_{4}^{2,0}(\ell_{0})=0, \phi_{4}^{1,1}(\ell_{0})=0, \phi_{4}^{0,2}(\ell_{0})=0.
\end{equation}
The proof goes as follows: we first compute the affine tangent space $T_{[\ell_{0}]}\F(X)$ of $\F(X)$ at the point $[\ell_{0}]$. Then we compute the affine tangent space $T_{[\ell_{0}]}\M(X)$ of $\M(X)$ at the point $[\ell_{0}]$. Afterwards we prove that the triple lines on $X$ are singular points of the curve $\M(X)$. Finally we prove that the singular points of the curve $\M(X)$ correspond to triple lines on $X$.\newline
\item[1.] \textit{Tangent space of $\F(X)$ at the point $[\ell_{0}]$.}
On the affine chart $p_{0,1}=1$, a point in $\mathbb{C}^{6}$ with coordinates $(p_{i,j})_{i,j}$ belongs to the tangent space $T_{[\ell_{0}]}\F(X)$ if and only if $\displaystyle\sum_{i=2}^{4} (-p_{1,i}t_{0}+p_{0,i}t_{1})\left(t_{0}^{2}\phi^{2,0}_{i}(\ell_{0})+t_{0}t_{1}\phi^{1,1}_{i}(\ell_{0})+t_{1}^{2}\phi^{0,2}_{i}(\ell_{0})\right)=0$ for all $(t_{0}:t_{1})\in\mathbb{P}^{1}$. Setting the coefficients of $t_{0}^{3}, t_{0}^{2}t_{1}, t_{0}t_{1}^{2}$ and $t_{1}^{3}$ each equal to zero one can see that for such a point the equality 
\begin{equation*}
\underbrace{
\begin{pmatrix}
0 & 0 & \phi_{2}^{2,0}(\ell_{0}) & \phi_{3}^{2,0}(\ell_{0})  \\ 
\phi_{2}^{2,0}(\ell_{0}) & \phi_{3}^{2,0}(\ell_{0})  & \phi_{2}^{1,1}(\ell_{0}) & \phi_{3}^{1,1}(\ell_{0}) \\
\phi_{2}^{1,1}(\ell_{0}) & \phi_{3}^{1,1}(\ell_{0})  & \phi_{2}^{0,2}(\ell_{0}) & \phi_{3}^{0,2}(\ell_{0})  \\
\phi_{2}^{0,2}(\ell_{0}) & \phi_{3}^{0,2}(\ell_{0})  & 0 & 0 
\end{pmatrix} 
}_{ A}
\begin{pmatrix}
p_{0,2} \\ 
p_{0,3} \\ 
-p_{1,2} \\ 
-p_{1,3} \\ 
\end{pmatrix} 
=\begin{pmatrix}
0 \\ 
0 \\ 
0 \\ 
0\\
\end{pmatrix}
\end{equation*}
holds (see Equation \eqref{eq:phi4}). By swapping two columns of the matrix $A$ we get the resultant
$$\det\begin{pmatrix}
\phi_{2}^{2,0}(\ell_{0}) & 0 & \phi_{3}^{2,0}(\ell_{0}) & 0  \\ 
\phi_{2}^{1,1}(\ell_{0}) & \phi_{2}^{2,0}(\ell_{0})  & \phi_{3}^{1,1}(\ell_{0}) & \phi_{3}^{2,0}(\ell_{0}) \\
\phi_{2}^{0,2}(\ell_{0}) & \phi_{2}^{1,1}(\ell_{0})  & \phi_{3}^{0,2}(\ell_{0}) & \phi_{3}^{1,1}(\ell_{0})  \\
0 & \phi_{2}^{0,2}(\ell_{0})  & 0 & \phi_{3}^{0,2}(\ell_{0})\\
\end{pmatrix}. $$
For any point $p_{0}\in\ell_{0}\subset X$ with coordinates $(t_{0}:t_{1}:0:0:0)$ in the affine chart $p_{0,1}=1$ the polynomials $\dfrac{\partial F}{\partial x_{i}}(p_{0})=\displaystyle\sum_{j+k=2}t_{0}^{j}t_{1}^{k}\phi_{i}^{j,k}(\ell_{0})$ vanish for $i\in\lbrace 0,1\rbrace$, and using Equation \eqref{eq:phi4} so also does $\dfrac{\partial F}{\partial x_{4}}(p_{0})=\displaystyle\sum_{j+k=2}t_{0}^{j}t_{1}^{k}\phi_{4}^{j,k}(\ell_{0})$. Therefore the equations $\dfrac{\partial F}{\partial x_{i}}(p_{0})=\displaystyle\sum_{j+k=2}t_{0}^{j}t_{1}^{k}\phi_{i}^{j,k}(\ell_{0})=0$ do not have a common root for $i\in\lbrace 2,3\rbrace$, otherwise the cubic threefold $X$ would be singular. Thus 
\begin{equation}\label{resultant}
\det(A)\neq 0
\end{equation}
and the tangent space to $\F(X)$ at the point $[\ell_{0}]$ is given by
\begin{equation*}
T_{[\ell_{0}]}\F(X)=\lbrace p_{0,2}=0,~p_{0,3}=0,~p_{1,2}=0,~p_{1,3}=0\rbrace
\end{equation*}
in the affine chart $p_{0,1}=1$.\newline
\item[2.] \textit{Tangent space of $\M(X)$ at the point $[\ell_{0}]$.}
The tangent space $T_{[\ell_{0}]}\M(X)$ is given by the intersection of the tangent space $T_{[\ell_{0}]}\F(X)$ with the hypersurface defined by $$\displaystyle\sum_{0\leq i < 2 \leq j\leq 4}p_{i,j}\dfrac{\partial \m}{\partial p_{i,j}}(\ell_{0})=0.$$ We differentiate the determinant in \eqref{det} with respect to $p_{i,j}$ and by using Equation \eqref{eq:phi4} we get
\begin{equation}\label{p04}
\det
\begin{pmatrix}
\phi^{2,0}_{2}(\ell_{0})&  \phi^{2,0}_{3}(\ell_{0}) & p_{0,4}\dfrac{\partial \phi^{2,0}_{4}}{\partial p_{0,4}}(\ell_{0}) + p_{1,4}\dfrac{\partial \phi^{2,0}_{4}}{\partial p_{1,4}}(\ell_{0})  \\ 
\phi^{1,1}_{2}(\ell_{0}) & \phi^{1,1}_{3}(\ell_{0}) & p_{0,4}\dfrac{\partial \phi^{1,1}_{4}}{\partial p_{0,4}}(\ell_{0}) + p_{1,4}\dfrac{\partial \phi^{1,1}_{4}}{\partial p_{1,4}}(\ell_{0})  \\ 
\phi^{0,2}_{2}(\ell_{0}) & \phi^{0,2}_{3}(\ell_{0})& p_{0,4}\dfrac{\partial \phi^{0,2}_{4}}{\partial p_{0,4}}(\ell_{0}) + p_{1,4}\dfrac{\partial \phi^{0,2}_{4}}{\partial p_{1,4}}(\ell_{0})\\
\end{pmatrix}
= 0.
\end{equation}
Let us explicit Equation (\ref{p04}). For any point $p\in\ell\subset X$ with coordinates $t_{0}v_{0}+t_{1}v_{1}$ in the affine chart $p_{0,1}=1$ one can write $F(p)=F(t_{0}v_{0}+t_{1}v_{1})$.
Differentiate Equation \eqref{NUM1} with respect to $p_{u,4}$ for $u\in\lbrace 0,1\rbrace$ and using Equation \eqref{NUM2} we have
\begin{equation*}
\sum_{j+k=2}t_{0}^{j}t_{1}^{k}\dfrac{\partial \phi_{4}^{j,k}}{\partial p_{u,4}}(\ell)=(-1)^{u}t_{0}^{u}t_{1}^{1-u}\sum_{k+l=1}t_{0}^{k}t_{1}^{l}\phi_{4,4}^{k,l}(\ell). 
\end{equation*}
By identifying the homogeneous components in $t_{0}$ and $t_{1}$ we get
\begin{equation*}
\dfrac{\partial \phi^{j,k}_{4}}{\partial p_{u,4}}(\ell)=\phi_{4,4}^{j,k-1}(\ell),\quad \dfrac{\partial \phi^{j,k}}{\partial p_{u,4}}(\ell)=-\phi_{4,4}^{j-1,k}(\ell)
\end{equation*}
with $\phi_{4,4}^{2,-1}(\ell)=0, \phi_{4,4}^{-1,2}(\ell)=0$.
Then Equation \eqref{p04} becomes
\begin{equation*}
\det
\begin{pmatrix}
\phi^{2,0}_{2}(\ell_{0})&  \phi^{2,0}_{3}(\ell_{0}) &  -p_{1,4}\phi_{4,4}^{1,0}(\ell_{0})  \\ 
\phi^{1,1}_{2}(\ell_{0}) & \phi^{1,1}_{3}(\ell_{0}) & p_{0,4}\phi_{4,4}^{1,0}(\ell_{0}) - p_{1,4} \phi_{4,4}^{0,1}(\ell_{0})  \\ 
\phi^{0,2}_{2}(\ell_{0}) & \phi^{0,2}_{3}(\ell_{0})& p_{0,4}\phi_{4,4}^{0,1}(\ell_{0}) \\
\end{pmatrix}= \phi(p_{0,4},p_{1,4}) = 0.
\end{equation*}
Therefore the tangent space of the curve $\M(X)$ at the point $[\ell_{0}]$ is given by
$$T_{[\ell_{0}]}\M(X) = \lbrace p_{0,2}=0,p_{0,3}=0,p_{1,2}=0,p_{1,3}=0,\phi(p_{0,4},p_{1,4}) = 0\rbrace$$ in the affine chart $p_{0,1}=1$. We now prove that the triple lines on $X$ are singular points of $\M(X)$.\newline
\item[3.] \textit{Triple lines and the curve $\M(X)$ of lines of the second type.} 
From Equation \eqref{droitetriple} the double line $[\ell_{0}]\in\M(X)$ is a triple line if and only if
\begin{equation*}
\left\{
    \begin{array}{ll}
        \phi_{4,4}^{1,0}(\ell_{0})=0 &  \\
        \phi_{4,4}^{0,1}(\ell_{0})=0 & 
    \end{array}
\right.
\end{equation*}
which implies that $\phi(p_{0,4},p_{1,4})=0$. Hence for $[\ell_{0}]\in\M(X)$ a triple line, the tangent space of the curve $\M(X)$ at the point $[\ell_{0}]$ is given by $$T_{[\ell_{0}]}\M(X) = \lbrace p_{0,2}=0, p_{0,3}=0, p_{1,2}=0, p_{1,3}=0\rbrace$$ in the affine chart $p_{0,1}=1$. Its dimension is strictly greater than one. Therefore the triple lines on $X$ are singular points of $\M(X)$.\newline
\item[4.] \textit{Singular points of the curve $\M(X)$ of lines of the second type.}
We have seen that the curve of lines of the second type is given by 
$\M(X)=\lbrace \phi^{3,0}(\ell)=0, \phi^{2,1}(\ell)=0, \phi^{1,2}(\ell)=0, \phi^{0,3}(\ell)=0, \m(\ell)=0\rbrace$ in the affine chart $p_{0,1}=1$. We are going to compute the Jacobian matrix of this curve at the point $[\ell_{0}]$. Keeping notation as above, we have for any point $p\in\ell\subset X$ with coordinates $t_{0}v_{0}+t_{1}v_{1}$:
\begin{equation}\label{NUM4}
F(p)=F(t_{0}v_{0}+t_{1}v_{1})=\displaystyle\sum_{i+j=3}t_{0}^{i}t_{1}^{j}\phi^{i,j}(\ell).
\end{equation}
Differentiating Equation (\ref{NUM4}) with respect to $p_{u,v}$ for $u\in\lbrace 0,1\rbrace$ and $v\in\lbrace 2,3,4\rbrace$ we get
\begin{equation*}
\displaystyle\sum_{i+j=3}t_{0}^{i}t_{1}^{j}\dfrac{\partial \phi^{i,j}}{\partial p_{u,v}}(\ell)=(-1)^{u}t_{0}^{u}t_{1}^{1-u}\sum_{j+k=2}t_{0}^{j}t_{1}^{k}\phi_{v}^{j,k}(\ell).
\end{equation*}
By identifying the homogeneous components in $t_{0}$ and $t_{1}$ we get
\begin{equation*}
\dfrac{\partial \phi^{i,j}}{\partial p_{0,v}}(\ell)=\phi_{v}^{i,j-1}(\ell),\quad\dfrac{\partial \phi^{i,j}}{\partial p_{1,v}}(\ell)=-\phi_{v}^{i-1,j}(\ell)
\end{equation*}
with $\phi_{v}^{3,-1}(\ell)=0, \phi_{v}^{-1,3}(\ell)=0$. The Jacobian matrix of $\M(X)$ at the point $[\ell_{0}]$ is thus
$$
\begin{pmatrix}
0 & 0 & 0 & -\phi_{2}^{2,0}(\ell_{0}) & -\phi_{3}^{2,0}(\ell_{0}) & 0 \\ 
\phi_{2}^{2,0}(\ell_{0}) & \phi_{3}^{2,0}(\ell_{0}) & 0 & -\phi_{2}^{1,1}(\ell_{0}) & -\phi_{3}^{1,1}(\ell_{0}) & 0  \\
\phi_{2}^{1,1}(\ell_{0}) & \phi_{3}^{1,1}(\ell_{0}) & 0 & -\phi_{2}^{0,2}(\ell_{0}) & -\phi_{3}^{0,2}(\ell_{0}) &  0  \\
\phi_{2}^{0,2}(\ell_{0}) & \phi_{3}^{0,2}(\ell_{0}) & 0 & 0 & 0 & 0 \\
\dfrac{\partial \m}{\partial p_{0,2}}(\ell_{0}) & \dfrac{\partial \m}{\partial p_{0,3}}(\ell_{0}) & \dfrac{\partial \m}{\partial p_{0,4}}(\ell_{0}) & \dfrac{\partial \m}{\partial p_{1,2}}(\ell_{0}) & \dfrac{\partial \m}{\partial p_{1,3}}(\ell_{0}) & \dfrac{\partial \m}{\partial p_{1,4}}(\ell_{0})\\
\end{pmatrix}
$$ 
using again Equation \eqref{eq:phi4}. Denote by $M_{i}$ its $5\times 5$ minors where $i$ corresponds to the omitted column. The minors $M_{1}, M_{2}, M_{4}, M_{5}$ vanish and the singular points of the curve $\M(X)$ are the points $[\ell_{0}]\in\M(X)$ for which $M_{3}=0,M_{6}=0$ with
\begin{eqnarray*}
 M_{3} &=& \dfrac{\partial \m}{\partial p_{1,4}}(\ell_{0})\det
\begin{pmatrix}
0 & 0 & -\phi_{2}^{2,0}(\ell_{0}) & -\phi_{3}^{2,0}(\ell_{0}) \\ 
\phi_{2}^{2,0}(\ell_{0}) & \phi_{3}^{2,0}(\ell_{0}) & -\phi_{2}^{1,1}(\ell_{0}) & -\phi_{3}^{1,1}(\ell_{0}) \\
\phi_{2}^{1,1}(\ell_{0}) & \phi_{3}^{1,1}(\ell_{0}) & -\phi_{2}^{0,2}(\ell_{0}) & -\phi_{3}^{0,2}(\ell_{0})  \\
\phi_{2}^{0,2}(\ell_{0}) & \phi_{3}^{0,2}(\ell_{0}) & 0 & 0 \\ 
\end{pmatrix}\\
M_{6} &=& \dfrac{\partial \m}{\partial p_{0,4}}(\ell_{0})\det 
\begin{pmatrix}
0 & 0 & -\phi_{2}^{2,0}(\ell_{0}) & -\phi_{3}^{2,0}(\ell_{0}) \\ 
\phi_{2}^{2,0}(\ell_{0}) & \phi_{3}^{2,0}(\ell_{0}) & -\phi_{2}^{1,1}(\ell_{0}) & -\phi_{3}^{1,1}(\ell_{0}) \\
\phi_{2}^{1,1}(\ell_{0}) & \phi_{3}^{1,1}(\ell_{0}) & -\phi_{2}^{0,2}(\ell_{0}) & -\phi_{3}^{0,2}(\ell_{0})  \\
\phi_{2}^{0,2}(\ell_{0}) & \phi_{3}^{0,2}(\ell_{0}) & 0 & 0
\end{pmatrix}.       
\end{eqnarray*}
Since the above determinant cannot vanish (see \eqref{resultant}) then one can see that the singular points of $\M(X)$ are the points $[\ell_{0}]$ that satisfy
\begin{equation*}
\left\{
    \begin{array}{ll}
  \dfrac{\partial \m}{\partial p_{1,4}}(\ell_{0}) =\det \begin{pmatrix}
\phi_{2}^{2,0}(\ell_{0}) & \phi_{3}^{2,0}(\ell_{0}) & -\phi_{4,4}^{1,0}(\ell_{0})  \\ 
\phi_{2}^{1,1}(\ell_{0}) & \phi_{3}^{1,1}(\ell_{0}) & -\phi_{4,4}^{0,1}(\ell_{0})  \\
\phi_{2}^{0,2}(\ell_{0}) & \phi_{3}^{0,2}(\ell_{0}) & 0\\
\end{pmatrix} = 0\\     
&\\
\dfrac{\partial \m}{\partial p_{0,4}}(\ell_{0}) =\det \begin{pmatrix}
\phi_{2}^{2,0}(\ell_{0}) & \phi_{3}^{2,0}(\ell_{0}) & 0  \\ 
\phi_{2}^{1,1}(\ell_{0}) & \phi_{3}^{1,1}(\ell_{0}) & \phi_{4,4}^{1,0}(\ell_{0})  \\
\phi_{2}^{0,2}(\ell_{0}) & \phi_{3}^{0,2}(\ell_{0}) & \phi_{4,4}^{0,1}(\ell_{0}) \\ 
\end{pmatrix}=0. \\
\end{array}
\right.
\end{equation*}
So we have
\begin{equation}\label{phi44}
\underbrace{
\begin{pmatrix}
-M_{13} & M_{23}  \\ 
-M_{23} & M_{33} \\ 
\end{pmatrix}
}_{B}
\begin{pmatrix}
\phi_{4,4}^{1,0}(\ell_{0})  \\ 
\phi_{4,4}^{0,1}(\ell_{0}) \\ 
\end{pmatrix}
=
\begin{pmatrix}
0 \\ 
0\\
\end{pmatrix}
\end{equation}
where $M_{13}, M_{23}$ and $M_{33}$ are the $2\times 2$ minors of the matrix 
$$\begin{pmatrix}
\phi^{2,0}_{2}(\ell_{0})& \phi^{2,0}_{3}(\ell_{0}) & \phi^{2,0}_{4}(\ell_{0}) \\ 
\phi^{1,1}_{2}(\ell_{0}) & \phi^{1,1}_{3}(\ell_{0}) & \phi^{1,1}_{4}(\ell_{0}) \\ 
\phi^{0,2}_{2}(\ell_{0}) & \phi^{0,2}_{3}(\ell_{0}) & \phi^{0,2}_{4}(\ell_{0}) \\
\end{pmatrix}. $$ 
This matrix is of rank two for if it was of rank one then  one could find a common root of the equations $\dfrac{\partial F}{\partial x_{i}}(p_{0})=\displaystyle\sum_{j+k=2}t_{0}^{j}t_{1}^{k}\phi_{i}^{j,k}(\ell_{0})=0$ for $i\in\lbrace 2,3,4\rbrace$ with $p_{0}\in\ell_{0}$ a point with coordinates $(t_{0}:t_{1}:0:0:0)$ in the affine chart $p_{0,1}=1$. Let us consider Equation \eqref{phi44}. If $\det(B)=0$ we can assume that a column is a multiple of the other one. Then one could find a common root of $\dfrac{\partial F}{\partial x_{i}}(p_{0}) = \displaystyle\sum_{j+k=2}t_{0}^{j}t_{1}^{k}\phi_{i}^{j,k}(\ell_{0})=0$ for $i\in\lbrace 2,3\rbrace$ and $X$ would be singular using Equation \eqref{eq:phi4}. Hence $\det(B)\neq 0$ and Equation \eqref{phi44} holds if and only if $\left(\phi_{4,4}^{1,0}(\ell_{0})=0,\phi_{4,4}^{0,1}(\ell_{0})\right)$ vanishes, which are the necessary and sufficient conditions for the line $[\ell_{0}]\in \M(X)$ to be a triple line. Therefore the singular points of $\M(X)$ are the triple lines on $X$. 
\end{proof}

As a consequence of Theorem \ref{th1} one can deduce that the curve $\M(X)$ of lines of the second type is reduced.

\begin{remark}
Let $\ell_{0}$ be a line on $X$ given by $x_{2}=0, x_{3}=0, x_{4}=0$. Then the equation of $X$ may be written 
\begin{equation*}
F(x_{0},...,x_{4}) = x_{2}q_{2}(x_{0},...,x_{4}) + x_{3}q_{3}(x_{0},...,x_{4}) + x_{4}q_{4}(x_{0},...,x_{4})=0
\end{equation*} 
where $q_{i}(x_{0},...,x_{4})$ are homogeneous polynomials of degree two. If moreover $\ell_{0}$ is a line of the second type so that the plane $P={\rm{span}}(\ell_{0},v_{2})$ tangent to $X$ in all points of $\ell_{0}$ is such that $\alpha_{2}=0, \alpha_{3}=0$ and $\alpha_{4}=1$ then the equation of $X$ may take the form
\begin{equation*}
F(x_{0},...,x_{4}) = x_{2}q_{2}(x_{0},...,x_{4}) + x_{3}q_{3}(x_{0},...,x_{4}) + x_{4}^{2}l(x_{0},...,x_{4})=0
\end{equation*} 
where $l(x_{0},...,x_{4})$ is a linear polynomial. Set $l(x_{0},...,x_{4})=a_{0}x_{0}+a_{1}x_{1}+a_{2}x_{2}+a_{3}x_{3}+a_{4}x_{4}$. Using the techniques of  \cite{murre1972algebraic} we get $$a_{0} = \dfrac{1}{2}\phi_{4,4}^{1,0}(\ell_{0})~~~~~~\mbox{and}~~~~~~ a_{1} = \dfrac{1}{2}\phi_{4,4}^{0,1}(\ell_{0}).$$
So for $\ell_{0}$ a triple line on $X$ we have $a_{0}=0,a_{1}=0$. As corrected in \cite[p.~ 17]{lahoz2021geometry}, Murre erroneously wrote $l(x_{0},x_{1})=a_{0}x_{0}+a_{1}x_{1}$ for the linear polynomial instead in equation (13) of \cite[p.~ 167]{murre1972algebraic}, implying that for $\ell_{0}\subset X$ a triple line this polynomial vanishes and $X$ contains the plane $\lbrace x_{2}=0, x_{3}=0\rbrace$: the cubic threefold $X$ would therefore be singular. Whereas for $\ell_{0}$ a triple line on $X$ the equation of $X$ may be written \begin{equation*}
F(x_{0},...,x_{4}) = x_{2}q_{2}(x_{0},...,x_{4}) + x_{3}q_{3}(x_{0},...,x_{4}) + kx_{4}^{3}=0
\end{equation*}
with $k\neq 0$ and the cubic threefold $X$ remains smooth.
\end{remark}

\section{Counting triple lines on the Fermat cubic}

In this section we prove the following proposition. \Fermat*

\begin{proof} 

In the affine chart $p_{0,1}=1$ the set $\M(F_{4})$ of lines of the second type of the Fermat cubic is a non smooth curve given by the following equations:
\begin{eqnarray*}
p_{1,2}^3 + p_{1,3}^3 + p_{1,4}^3 - 1&=&0\\
p_{0,2}p_{1,2}^{2}+p_{0,3}p_{1,3}^2 + p_{0,4}p_{1,4}^2&=&0\\
p_{0,2}^{2}p_{1,2}+p_{0,3}^2p_{1,3} +p_{0,4}^2p_{1,4}&=&0\\
p_{0,2}^3+p_{0,3}^3 + p_{0,4}^3 + 1&=&0\\
(p_{0,4}p_{1,3}-p_{0,3}p_{1,4})(p_{0,4}p_{1,2}-p_{0,2}p_{1,4})(p_{0,3}p_{1,2}-p_{0,2}p_{1,3})&=&0
\end{eqnarray*}
with
\begin{equation*}
(p_{0,4}p_{1,3}-p_{0,3}p_{1,4})(p_{0,4}p_{1,2}-p_{0,2}p_{1,4})(p_{0,3}p_{1,2} -p_{0,2}p_{1,3}) = \m(\ell)
\end{equation*}
its local equation in the Fano surface $\F(F_{4})$. Consider the quadrics $$Q_{1}=p_{0,4}p_{1,3}-p_{0,3}p_{1,4},~~Q_{2}=p_{0,4}p_{1,2}-p_{0,2}p_{1,4},~~Q_{3}=p_{0,3}p_{1,2} -p_{0,2}p_{1,3}.$$
We note that the intersection of the Fano surface $\F(F_{4})$ with $Q_{1}, Q_{2}$ and $Q_{3}$, denoted by $\M_{1}(F_{4}), \M_{2}(F_{4})$ and $\M_{3}(F_{4})$ respectively, are smooth curves that correspond to the irreducible components of $\M(F_{4})$.
\item The intersection points $[\ell]$ of the curves $\M_{1}(F_{4})$ and $\M_{2}(F_{4})$ are given by the equations:
\begin{center}
$p_{0,2}p_{1,3}^3-p_{0,2},~~p_{1,3}^4-p_{1,3},~~p_{0,2}^3+p_{1,3}^3,~~p_{0,3}^3- p_{1,3}^3+1,~~p_{1,2}^3+p_{1,3}^3-1,\newline p_{0,2}p_{0,3},p_{0,2}p_{1,2},~~p_{0,3}p_{1,3},~~p_{1,2}p_{1,3},~~p_{0,4},~~p_{1,4}$.
\end{center}
If $p_{1,3}=0$ then we have 9 intersection points given by the following equations:
$$p_{0,3}^3=-1,~~p_{1,2}^3=1,~~p_{0,2}=0,~~p_{0,4}=0,~~p_{1,4}=0.$$
We get $v_{2}=(0:0:0:0:1)$.
We have thus $f(t_{0}v_{0}+t_{1}v_{1}+t_{2}v_{2})=t_{2}^{3}$
implying that $\ell$ is a triple line on the Fermat cubic. Also the Jacobian matrix of the curve $\M(F_{4})$ at the point $[\ell]$ is not of rank 5; $[\ell]$ is a singular point of $\M(F_{4})$. We get 9 triple lines with coordinates $(0, p_{0,3}, 0, p_{1,2}, 0, 0)$ such that $p_{0,3}^{3}=-1$ and  $p_{1,2}^{3}=1$.
\newline
If $p_{1,3}\neq 0$ then we have 9 intersection points given by the equations:
$$p_{0,2}^3=-1,~~p_{1,3}^3=1,~~p_{0,3}=0,~~p_{0,4}=0,~~p_{1,2}=0,~~p_{1,4}=0.$$
They correspond to triple lines with coordinates $(p_{0,2},0,0,0,p_{1,3},0)$ with $p_{0,2}^{3}=-1$ and  $p_{1,3}^{3}=1$.

Similarly we find the intersection points $[\ell]$ of the curves $\M_{1}(F_{4})$ and $\M_{3}(F_{4})$, $\M_{2}(F_{4})$ and $\M_{3}(F_{4})$. We get 18 triple lines at each intersection. We note that the intersection of the three curves $\M_{1}(F_{4}), \M_{2}(F_{4})$ and $\M_{3}(F_{4})$ is empty. Therefore there are 54 triple lines in the affine chart $p_{0,1}=1$.

We do the same computations to count the triple lines on the Fermat cubic in the other affine charts of $\mathbb{G}(1,4)$ by using SageMath \cite{sagemath}. However a triple line on $F_{4}$ can be viewed in many affine charts and then counted several times. In order to avoid such a situation we use the Plücker stratification (see \cite{boissiere2007counting}). In the stratum $p_{0,1}=0, p_{0,2}=1$ we get 36 triple lines, we get 18 triple lines in the stratum $p_{0,1}=p_{0,2}=0, p_{0,3}=1$ and in the stratum $p_{0,1}=...=p_{0,4}=0, p_{1,2}=1$, we get
9 triple lines in the stratum $p_{0,1}=...=p_{1,2}=0, p_{1,3}=1$ whereas the other strata contain no triple line. So there are exactly 135 triple lines on the Fermat cubic $F_{4}\subset \mathbb{P}^{4}$.
\end{proof}

\begin{remark} 
It is clear that the curve $\M(F_{4})$ of lines of the second type of the Fermat cubic $F_{4}\subset \mathbb{P}^{4}$ is reducible. The Fano surface $\F(F_{4})$ of the Fermat cubic contains 30 elliptic curves that intersect in 135 points \cite{roulleau2011fano}. These intersection points are exactly the triple lines on the Fermat cubic. We wish to point out that the elliptic curves of the Fano surface $\F(F_{4})$ are exactly the irreducible components of the curve $\M(F_{4})$ of lines of the second type.
\end{remark}

\bibliographystyle{amsalpha}
\bibliography{mabiblioV2}

\end{document}